\documentclass[10pt]{amsart}
\usepackage{amsmath, amsfonts, amsthm, amssymb}
\usepackage[a4paper, width=16cm, top=30mm, bottom=30mm]{geometry}
\usepackage{mathrsfs}
\usepackage{subfigure}
\usepackage{stmaryrd}
\usepackage{verbatim}
\usepackage{hyperref}
\usepackage{color}
\usepackage{blindtext}
\usepackage{bbm}
\usepackage{ulem}
\usepackage{enumerate}
\numberwithin{equation}{section}
\newtheorem{theorem}{Theorem}[section]
\newtheorem{corollary}[theorem]{Corollary}
\newtheorem{lemma}[theorem]{Lemma}
\newtheorem{proposition}[theorem]{Proposition}

\theoremstyle{definition}
\newtheorem{definition}[theorem]{Definition}
\newtheorem{remark}[theorem]{Remark}
\newtheorem{assumption}[theorem]{Assumption}

\newcommand{\I}{\mathtt{i}}

\newcommand{\RR}{\mathbb{R}}
\newcommand{\PP}{\mathbb{P}}
\newcommand{\II}{\mathbb{I}}

\newcommand{\EE}{\mathbb{E}}
\newcommand{\VV}{\mathbb{V}}

\newcommand{\Cc}{\mathcal{C}}
\newcommand{\Mm}{\mathcal{M}}
\newcommand{\mm}{\mathfrak{m}}
\newcommand{\Oo}{\mathcal{O}}
\newcommand{\Ii}{\mathcal{I}}

\newcommand{\Tt}{\mathcal{T}}
\newcommand{\NN}{\mathbb{N}}

\newcommand{\D}{\mathrm{d}}

\newcommand{\xx}{\mathrm{x}}
\newcommand{\yy}{\mathrm{y}}
\newcommand{\E}{\mathrm{e}}
\newcommand{\Bs}{\mathscr{B}}
\newcommand{\Es}{\mathscr{E}}

\newcommand{\LTwo}{\mathrm{L}^2}
\newcommand{\Ff}{\mathscr{F}}

\newcommand{\BV}{\mathrm{BV}}

\newcommand{\eps}{\varepsilon}
\newcommand{\zxy}{\mathrm{z}^{x}_{y}}
\newcommand{\zxyS}{\mathrm{z}^{x^*}_{y^*}}
\newcommand{\zxyO}{\mathrm{z}^{x_1}_{y_1}}
\newcommand{\zxyT}{\mathrm{z}^{x_2}_{y_2}}
\newcommand{\be}{\begin{equation}}
\newcommand{\ee}{\end{equation}}

\def\blue#1{\textcolor{black}{#1}}

\begin{document}
\title{Pathwise large deviations for the Rough Bergomi model}

\author{Antoine Jacquier}
\address{Department of Mathematics, Imperial College London}
\email{a.jacquier@imperial.ac.uk}

\author{Mikko S. Pakkanen}
\address{Department of Mathematics, Imperial College London and CREATES, Aarhus University}
\email{m.pakkanen@imperial.ac.uk}

\author{Henry Stone}
\address{Department of Mathematics, Imperial College London}
\email{henry.stone15@imperial.ac.uk}

\date{\today}

\thanks{
AJ acknowledges financial support from the EPSRC First Grant EP/M008436/1. HS thanks the EPSRC CDT in Financial Computing and Analytics for financial support. MSP acknowledges partial support from CREATES (DNRF78), funded by the Danish National Research Foundation.
}
\maketitle
\begin{abstract}
Introduced recently in mathematical finance by Bayer, Friz and Gatheral~\cite{BFG16}, 
the rough Bergomi model has proved particularly efficient to calibrate option markets.
We investigate here some of its probabilistic properties, in particular proving a pathwise large deviations principle for a small-noise version of the model.
The exponential function (continuous but superlinear) as well as the drift appearing in the volatility process
fall beyond the scope of existing results, and a dedicated analysis is needed.
\end{abstract}


\section{Introduction}
The extension of the Black-Scholes model, in which volatility is assumed to be constant, to the case where the volatility is stochastic has proved successful in explaining certain phenomena observed in option price data, 
in particular the implied volatility smile. 
The main shortcoming of such stochastic volatility models, however, is that they are unable to capture the true steepness of the implied volatility smile close to maturity. 
While choosing to add jumps to stock price models, for example modelling the stock price process as an exponential L\'evy process, does indeed produce steeper implied volatility smiles~\cite{FF12}, 
the question of the presence of jumps in stock price processes remains controversial~\cite{BST15,COP14}.

As an alternative to exponential L\'evy and classical stochastic volatility  models, one may choose a fractional Brownian motion, or a process with similar fine properties, to drive the volatility process, rather than a standard Brownian motion. Since volatility is neither directly observable nor tradable, the issue of arbitrage that is sometimes associated to fractional Brownian motion does not arise in this case.
A fractional Brownian motion is a centred Gaussian process, whose covariance structure depends on a Hurst parameter $H\in (0,1) $. 
If $H\in (0,1/2)$, then the fractional Brownian motion has negatively correlated increments and `rough' sample paths, and if $H \in (1/2, 1)$ then it has positively correlated increments and `smooth' sample paths, 
when compared to a standard Brownian motion, which is recovered by taking $H=1/2$.
There has been a resurgent interest in fractional Brownian motion and related processes within the mathematical finance community in recent years. 
In particular, Gatheral, Jaisson and Rosenbaum~\cite{GJR14} carried out an empirical study that suggests that the log volatility behaves at short time scales in a manner similar to a fractional Brownian motion, with Hurst parameter $H \approx 0.1$.
This finding is corroborated by Bennedsen, Lunde and Pakkanen~\cite{BLP16}, who study over a thousand individual US equities and find that the Hurst parameter~$H$ lies in $(0, 1/2)$ for each equity.
In addition, such so-called `rough' volatility models are able to capture the observed steepness of small-time implied volatility smiles and the term structure of at-the-money skew much better than classical stochastic volatility models. 

Following~\cite{GJR14}, Bayer, Friz and Gatheral~\cite{BFG16} propose the so-called rough Bergomi model, which they used to price options on integrated volatility and on the underlying itself. 
The advantage of their model is that it captures the rough behaviour of log volatility, 
in accordance with~\cite{BLP16, GJR14}, 
as well as fits observed implied volatility smiles better than traditional Markovian stochastic volatility models, most notably in the close-to-maturity case. The works \cite{AlosLeon, Fukasawa1, Fukasawa2} 
study the short-time behaviour of rough volatility models. 
Despite recent advances in simulation methods for the rough Bergomi model~\cite{BLP15,JMM17}, 
it is necessary to seek a more profound analytical understanding of the properties of this model. 
Specifically, in the present paper we prove pathwise large deviations for this model,
which allow to characterise its small-time behaviour.
Related results have been recently obtained by 
Bayer, Friz, Gulisashvili, Horvath and Stemper~\cite{BFGHS17},
and large deviations theory is now a common tool for such analysis
in standard stochastic volatility models~\cite{DFJV, FJ09, LDPBook, JKRM, JR}, 
and for their rough counterparts~\cite{BFGHS17, FZ15}.

In Section~\ref{Section: model and main results}, we present the correlated rough Bergomi model, together with its main properties, and lay out the main results of the paper; specifically a small-time large deviations principle for a rescaled version of the normalised process.
In Section~\ref{sec:GaussLDP}, we present several elements from the theory of Gaussian measures and large deviations that are required to prove the main results of the paper.
In Section~\ref{Section: proof of main results}, we give the proofs of the main results,
and Section~\ref{sec: uncorrel} elucidates the analogous large deviations result for the uncorrelated rough Bergomi model.  

\textbf{Notations:}
The notation $\LTwo:=L^2(\Tt, \RR)$ stands for the space of real-valued square integrable functions on some index set~$\Tt$,
and $\Cc^d := \Cc(\Tt, \RR^d)$ the space of $\RR^d$-valued continuous functions on~$\Tt$.
\blue{We shall fix $\Tt = [0,1]$ for the rest of this paper, although our results can be easily adapted for general interval $[0,T]$.}
We shall further denote~$\BV$ the space of paths of finite variations on~$\Tt$, 
and $\RR_+:=[0,\infty)$.
For two paths~$x$ and~$y$ belonging to~$\Cc = \Cc^1$, we denote by~$\zxy$ the two-dimensional path 
$(x,y)^\top$.
Now, $\II(\zxy)(t)$ represents the integral (whenever well-defined) $\int_0^{t} \sqrt{x(s)}\D y(s)$;
the expression~$\II(\zxy)$ shall be used whenever the integral is taken over the whole time period~$[0,1]$,
 and $x\cdot y := \int_{0}^{1} x(s) \D y(s)$.

\section{Model and main results}\label{Section: model and main results}
We assume a given filtered probability space $(\Omega, \mathscr{A}, (\Ff_t )_{t \ge 0}, \PP)$, 
where the filtration satisfies the usual conditions, and all stochastic processes here will be assumed to live on this probability space.

\subsection{Rough Bergomi Model and main properties}

Bayer, Friz and Gatheral~\cite{BFG16} introduce a non-Markovian generalisation of Bergomi's 
`second generation' stochastic volatility model, 
 which they dub `rough Bergomi' model. 
Let~$Z$ be the process defined pathwise as 
\begin{equation}\label{eq:SDEZ}
Z_t := \int_0^t K_\alpha(s,t)\D W_s,
\qquad\text{for any }t\geq 0,
\end{equation}
where $\alpha \in \left(-\frac{1}{2},0\right)$, $W$ a standard Brownian motion, 
and where the kernel 
$K_{\alpha}:\RR_+\times\RR_+ \to \RR_+$ reads
\begin{equation}\label{eq:K}
K_{\alpha}(s,t) := \eta \sqrt{2\alpha + 1}(t-s)^{\alpha}, 
\qquad \text{for all } 0\leq s<t,
\end{equation}
for some strictly positive constant~$\eta$.
Note that, for any $t\geq 0$, the map $s\mapsto K_\alpha(s,t)$ belongs to~$\LTwo$,
so that the stochastic integral~\eqref{eq:SDEZ} is well defined.
The rough Bergomi model, where $X$ is the log stock price process and $v$ is the volatility process, is then defined as
\begin{equation}\label{rBergomi}
\begin{array}{rll}
X_t & = \displaystyle - \frac{1}{2} \int_0^t v_s \D s + \int_0^t \sqrt{ v_s } \D B_s,
 \quad &  X_0 = 0 , \\
v_t &= \displaystyle v_0
\exp \left( Z_t -\frac{\eta^2}{2}t^{2 \alpha +1}  \right),
\quad & v_0 > 0,
\end{array}
\end{equation}
where the Brownian motion~$B$ is defined as $B := \rho W + \sqrt{1-\rho^2}W^\perp$ for $\rho \in [-1,1]$
and some standard Brownian motion~$W^\perp$ independent of $W$. 
The filtration~$(\Ff_t)_{t\geq 0}$ can here be taken as the one generated by the two-dimensional Brownian motion 
$(W,W^\perp)$.

\begin{proposition}\label{prop: (Z,B) properties}
The two-dimensional Gaussian process $(Z,B)$ is centred with covariance structure 
\begin{align*}
\mathrm{cov}\left(\begin{pmatrix}Z_t\\B_t\end{pmatrix}, \begin{pmatrix}Z_t\\B_t\end{pmatrix}\right)
& = 
\begin{pmatrix}
\eta^2 t^{2\alpha +1 } & \varrho t^{\alpha + 1} \\
\varrho t^{\alpha + 1} & t \\
\end{pmatrix},\\
\EE(Z_s Z_t) & = \int_0^{s \wedge t} K_{\alpha}(u,s)K_{\alpha}(u,t)\D u 
 = \frac{\eta^2 (2\alpha +1)}{\alpha + 1} (s \wedge t )^{1+ \alpha} (s \vee t )^{\alpha}{}_2 F_1 \left(1, -\alpha, 2 + \alpha, \frac{s \wedge t}{ s \vee t  } \right), 
\end{align*}
for any $s,t\geq 0$, with $ \varrho :=\frac{\rho \eta \sqrt{2\alpha +1}}{\alpha + 1}$
and ${}_2 F_1$ the Gauss hypergeometric function~\cite[Chapter~5, Section~9]{Olver}. 
\end{proposition}

\begin{proof}
Without loss of generality, assume first that $s<t$. This then implies that
$\EE(Z_s Z_t) = \eta^2 (2\alpha + 1) \int_0^s(t-u)^\alpha (s-u)^\alpha \D u
=
t^\alpha s^{1+\alpha} \int^1_0 (1-v)^\alpha (1 - sv/t)^\alpha \D v,
$
where the second equality follows from a change of variables.
Using Euler's integral representation of the Gauss hypergeometric function ${}_2 F_1$, it follows that 
$ \int_0^s(t-u)^\alpha (s-u)^\alpha \D u
=
\frac{1}{\alpha+1} {}_2 F_1 (-\alpha, 1 ; \alpha + 2; s/t ),
$ 
from which the proposition follows.
\end{proof}

Proposition \ref{prop: (Z,B) properties} implies in particular that the process $Z$ is not stationary, and that the following holds:
\begin{corollary}
The process $Z$ is 
$(\alpha + \frac{1}{2})$ self-similar: for any $a>0$,
the processes $(Z_{at})_{t \ge 0}$ and $(a^{\alpha + \frac{1}{2}} Z_t)_{t \ge 0}$ are equal in distribution.
\end{corollary}

Note then that the parameter~$\alpha$ determines both the local and long-term behaviour of~$Z$.

\begin{remark}
The process~$Z$ in~\eqref{eq:SDEZ} is the Holmgren-Riemann-Liouville fractional Brownian motion introduced 
by L\'evy~\cite{Levy}, modulo some constant scaling, rather than the more commonly known fractional Brownian motion
characterised by Mandelbrot and Van Ness~\cite[Definition 2.1]{MN68} as
$$
W^H_t = \frac{1}{\Gamma(H+1/2)} 
\left( \int_{- \infty }^0 ( (t-s)^{H-1/2} - (-s)^{H-1/2} ) \D \widetilde{W}_s
+ \int_0^t (t-s)^{H-1/2}\D \widetilde{W}_s \right),
$$
where~$\widetilde{W}$ is a standard Brownian motion, and~$\Gamma$ the standard Gamma function.
The Mandelbrot-Van Ness representation of $W^H_t$ requires the knowledge of~$\widetilde{W}$ from~$-\infty$ to~$t$; 
in comparison we only need to know~$W$ from~$0$ to~$t$ to compute the value of~$Z$.
Also both~$Z$ and~$W^H$ are self-similar, but~$W^H$ has stationary increments whereas the increments of $Z$ are non-stationary.
\end{remark}

\begin{proposition} 
The process $\log v$ has a modification whose trajectories are almost surely locally $\gamma$-H\"older continuous,
for all $\gamma \in \left(0, \alpha + \frac{1}{2} \right)$.
\end{proposition}

\begin{proof}
We first prove that~$Z$ has a modification whose trajectories are $\gamma$-H\"older continuous, for all $\gamma \in (0, \alpha + \frac{1}{2})$.
Firstly, 
$\EE(\vert Z_t - Z_s\vert^2) \le \eta^2(2\alpha + 1)  
\left( 
\int_s^t \vert t-u \vert ^{2\alpha} \D u 
+ \int_0^s \vert (t-u)^\alpha - (s-u)^\alpha \vert^2 \D u 
\right) 
=
\eta^2 \vert t-s\vert^{2\alpha + 1} + \eta^2(2\alpha + 1)\int_0^s \vert (t-u)^\alpha - (s-u)^\alpha \vert^2 \D u.  $ 
Following the change of variables $s-u = (t-s)y$ the integral becomes
$\vert t-s \vert^{2\alpha+1} \int^{\frac{s}{t-s}}_0 \vert (y+1)^\alpha - y^\alpha \vert ^2 \D y$, 
and hence $\int^{\frac{s}{t-s}}_0 \vert (y+1)^\alpha - y^\alpha \vert ^2 \D y$ is finite
since $\alpha \in (-\frac{1}{2},0)$.
Therefore there exists $K>0$ such that 
$\EE(\vert Z_t - Z_s\vert^2) \le K \vert t-s \vert^{2\alpha+1}$ for any $s,t\ge 0$.
Applying the Kolmogorov continuity theorem \cite[Theorem 3.22]{Kal02} 
then yields that the Gaussian process~$Z$ has a modification whose trajectories are locally $\gamma$-H\"older continuous where $\gamma \in (0, \alpha+ \frac{1}{2})$, thus proving the claim.
Now, for the process $\log v$, we have
\begin{align*}
\vert \log v_t - \log v_s \vert 
 & = \left|Z_t - \frac{\eta^2 }{2}t^{2 \alpha +1} -\left( Z_s - \frac{\eta^2}{2} s^{2 \alpha +1}\right)\right|\\
 & \le \vert Z_t -Z_s \vert + \frac{\eta^2}{2} \left| t^{2 \alpha + 1} - s^{ 2\alpha + 1}\right|
\le C \vert t-s \vert^{\gamma} + \frac{\eta^2}{2} \left|t^{2 \alpha + 1} - s^{ 2\alpha + 1}\right|,
\end{align*}
where $C$ is a strictly positive constant, and $\gamma \in (0, \alpha + 1/2)$. 
Since the map $t \mapsto t^{2 \alpha +1}$ is also locally $\gamma$-H\"older continuous 
for all $\gamma \in (0,2 \alpha + 1]$ and in particular for all $\gamma \in (0, \alpha + 1/2)$,
it follows that the process $ \log v $ has a modification with locally $\gamma$-H\"older continuous trajectories, for all $ \gamma \in (0, \alpha + \frac{1}{2})$.
\end{proof}

As a comparison, the fractional Brownian motion has sample paths that are $\gamma$-H\"older continuous 
for any $\gamma \in (0,H)$~\cite[Theorem 1.6.1]{BHOZ08}, so that
the rough Bergomi model also captures this roughness by identification $\alpha= H - 1/2$;
in particular these trajectories are rougher than those of the standard Brownian motion, for which $H=1/2$.


\subsection{Main results}

For any functions $\varphi_1, \varphi_2:\RR_+\times\RR_+\to\RR$, 
introduce the $\LTwo $ operators $\Ii^{\varphi_1}$ and $\Ii^{\varphi_1}_{\varphi_2} $ defined as 
\begin{equation}\label{eq:IiOper}
\Ii^{\varphi_1} f :=\int_0^{\cdot} \varphi_1(u,\cdot)f(u)\D u
\qquad \text{and } \qquad 
\Ii^{\varphi_1}_{\varphi_2} f := 
\begin{pmatrix}
\displaystyle \Ii^{\varphi_1} f\\
\displaystyle \Ii^{\varphi_2} f
\end{pmatrix}.
\end{equation}
Whenever the function~$\varphi$ is constant, say equal to~$c$, we shall write $\Ii^c$ without ambiguity.
We also define the space 
$\mathscr{H}^{\varphi_1}_{\varphi_2} := \left\{\Ii^{\varphi_1}_{\varphi_2} f: f \in \LTwo\right\}$
that
is clearly a Hilbert space once endowed with the inner product
$\left\langle \Ii^{\varphi_1}_{\varphi_2}f_1, \Ii^{\varphi_1}_{\varphi_2} f_2\right\rangle_{\mathscr{H}^{\varphi_1}_{\varphi_2}} := \langle f_1, f_2 \rangle_{\LTwo}$, 
where $\varphi_{1},\varphi_{2}$ are such that~$\Ii^{\varphi_1}_{\varphi_2}$ is injective, 
making the inner product well-defined.
For $t, \eps \ge 0$, now define the rescaled processes
\begin{equation}\label{def: rescaled rBerg processes}
X^{\eps}_t :=  \eps^{\beta}X_{\eps t}, \qquad
Z^{\eps}_t :=  \eps^{\beta /2}Z_t \overset{\textrm{d}}= Z_{\varepsilon t},  \qquad
v^{\eps}_t :=  \eps^{1 + \beta } v_0 \exp \left( Z^{\eps}_t -\frac{\eta^2}{2} (\eps t)^{\beta}  \right), \qquad
B^{\eps}_t :=  \eps^{\beta /2 }B_t,
\end{equation} 
where $\beta := 2 \alpha + 1 \in (0,1)$.
Note in particular that, for any $t,\eps\geq 0$, $Z^{\eps}_t$ and $Z_{\eps t}$ are equal in law,
and so are~$v^{\eps}_{t}$ and~$\eps^{1 + \beta} v_{\eps t}$,
which in turn implies that the following representation holds for any $t\geq 0$:
\begin{align}\label{eq:SDEXeps}
X^{\eps}_{t} & :=  \eps^{\beta}X_{\eps t}
\overset{\textrm{d}} = \eps^{\beta } \left( \int_0^{\eps t} \sqrt{v_s} \D B_s -\frac{1}{2} \int_0^{\eps t} v_s \D s \right) 
\overset{\textrm{d}} =  \eps^{\beta } \left( \int_0^{t} \sqrt{v_{\eps s  }   }\D B_{\eps s}
 - \frac{\eps}{2} \int_0^{t} v_{\eps s} \D s  \right) \nonumber\\
& \overset{\textrm{d}}  =  \int_0^{t} \sqrt{ \eps^{ 1 + 2 \beta } v_ {\eps s }  } \D B_s -\frac{1}{2} \int_0^{t} \eps^{ 1 + \beta } v_{ \eps s } \D s
\overset{\textrm{d}} = \int_0^{t} \sqrt{v^{\eps  }_s  } \D B^{\eps }_s -\frac{1}{2} \int_0^{t} v^{\eps  }_s \D s.
\end{align}

We now state the main result of this section, namely a pathwise large deviations principle 
for the sequence of rescaled processes~$(X^{\eps})_{\eps \ge 0 }$. 
We recall first some facts about large deviations on a real, separable Banach space $\left( \Es, \|\cdot\|_{ \Es} \right)$, following~\cite{DZ10} as our guide.

\begin{definition} 
A function  $\Lambda:\Es$ $\to [0,+ \infty]$ is said to be a rate function if it is lower semi-continuous, that is,
if, for all $\xx_0 \in \Es $,
$$
\liminf_{\xx \rightarrow \xx_0} \Lambda(\xx) \ge ~\Lambda(\xx_0).
$$
\end{definition}

\begin{definition}
A family of probability measures $(\mu_\eps)_{\eps\ge 0}$ on $(\Es, \Bs(\Es))$  
is said to satisfy a large deviations principle (LDP) as~$\eps$ tends to zero with speed $\eps^{-1}$ 
and rate function~$\Lambda$ if, for any $B \in \Bs(\Es)$,
\begin{equation}\label{eqn:LDPdefinition}
- \inf_{\xx \in B^{\circ}} \Lambda(\xx)
 \leq \liminf_{\eps \downarrow 0} \varepsilon \log \mu_\eps(B)
 \leq \limsup_{\eps \downarrow 0} \eps \log\mu_\eps(B)
 \leq - \inf_{\xx \in \overline{B}} \Lambda(\xx),
\end{equation}
where $\overline{B}$ and $B^\circ$ denote respectively the closure and the interior of~$B$.
\end{definition}

Correspondingly, a stochastic process  $(Y_\eps)_{\eps \ge 0}$ is said to satisfy a LDP as~$\eps$ tends to zero if the family of probability measures $(\PP(Y_\eps \in \cdot))_{\eps\ge 0} $ satisfies a LDP as~$\eps$ tends to zero.
Unless otherwise stated, all LDP here shall be as~$\eps$ tends to zero, so we shall drop this mention for simplicity.
To state our results, we now define the operator $\Mm: \Cc^2 \to \Cc(\Tt \times \RR_{+}, \RR_{+} \times \RR)$ as 
\begin{equation}\label{def: M operator}
(\Mm\zxy)(t, \eps)
:= \begin{pmatrix}
(\mm x)(t,\eps)\\
y(t)
\end{pmatrix}
\quad \textrm{ for all } t\in\Tt, \eps>0,
\end{equation}
where the operator $\mm:\Cc\to\Cc(\Tt \times \RR_+, \RR_+)$ is defined by
\begin{equation}\label{eq:M1}
(\mm x)(t,\eps)
:= v_0\eps^{1 + \beta}\exp\left( x(t) -  \frac{\eta^2}{2}(\eps t)^{\beta }\right),
\end{equation}
as well as the functions
$\Lambda^{*}, \Lambda: \Cc^2 \to\RR_+$ defined by
$$
\Lambda^{*}(\zxy):= \frac{1}{2}\|\zxy\|_{\mathscr{H}^{K_\alpha}_{\rho}}^2
\qquad\text{and}\qquad
\Lambda(\zxyO)
:=
\inf\Big\{ \Lambda^{*}(\zxyT): \zxyO = \Mm\zxyT \Big\}.
$$

\begin{theorem}\label{thm:rBergLDPthm}
The sequence $(X^{\eps})_{\eps \ge 0}$  satisfies a LDP on~$\Cc$ with speed $\eps^{-\beta}$ 
and rate function $\Lambda^X~:~\Cc~\to~[0, + \infty]$ defined as
$\Lambda^X(\varphi):= \inf \left\{\Lambda(\zxy): \varphi = \sqrt{x}\cdot y, 
 y\in\BV\cap\Cc\right\}$.
\end{theorem}

\begin{corollary}\label{LDPSmalltX}
The rescaled log stock price process $\left( t^{\beta} X_t \right)_{t \ge 0}$
satisfies a LDP on $\RR$ with speed $t^{-\beta}$ 
and rate function $\Lambda^X_1(u): = \inf \{\Lambda^X(\varphi) : \varphi(1) = u  \}$, $u \in \RR$.
\end{corollary}
\begin{proof}
Since $X^{\eps}_1$ and $\eps^{\beta} X_{\eps}$ are equal in law, 
$(\eps^{\beta} X_{\eps})_{\eps \ge 0 }$ 
satisfies a LDP with speed~$\eps^{-\beta}$ and rate function~$\Lambda^X_1$  
by Theorem~\ref{thm:rBergLDPthm} together with the \blue{contraction principle (Proposition \ref{contractionPrinciple} below).
Substituting}~$\eps$ \blue{with}~$t$ completes the proof.
\end{proof}
\begin{remark}
Recently, Forde and Zhang~\cite{FZ15} derived pathwise large deviations for rough volatility models,
with application (by scaling) to small-time asymptotics of the corresponding implied volatility.
The model they consider is of the following form, for the log stock price process:
\begin{equation*}
\left\{
\begin{array}{ll}
\D X_t = -\frac{1}{2}\sigma(Y_t)^2 \D t + \sigma(Y_t) \D B_t, \\
Y_t = W_t^H,
\end{array}
\right.
\end{equation*}
where $B$ is a standard Brownian motion, $W^H$ a (possibly correlated) fractional Brownian motion.
In order to prove LDP, they consider a small-noise version of the SDE above, namely:
\begin{equation*}
\left\{
\begin{array}{ll}
\D X_t^\eps = -\frac{1}{2}\eps\sigma(Y_t)^2 \D t + \sqrt{\eps}\sigma(Y_t) \D B_t, \\
Y_t^\eps = \eps^H W_t^H.
\end{array}
\right.
\end{equation*}
It is of course tempting to apply their results to the rough Bergomi model.
Unfortunately, the following intricacies make this impossible:
firstly, they assume the function~$\sigma$ to have at most linear growth, whereas it is of exponential growth in rough Bergomi;
secondly, their scaling assumption, allowing them to translate small-noise into small-time estimates 
crucially relies on the volatility process~$Y$ being driftless \cite[Equation~(4.4)]{FZ15},
which does not hold in rough Bergomi.
\end{remark}

\blue{Theorem \ref{thm:rBergLDPthm} and Corollary \ref{LDPSmalltX} have several} potential applications within mathematical finance, two of which we now outline. One application is to establish the small-time behaviour of \blue{implied volatility under} the rough Bergomi model. \blue{This would require, however, first showing that the stock price process in the rough Bergomi model is a true martingale, which is far from trivial and beyond the scope of the present paper.} 
\blue{Another potential} application is to variance reduction of \blue{Monte Carlo pricing of path-dependent options, following the importance sampling approach employed in \cite{MS2018IS,Rob10,S2013IS}}.

There is a degree of flexibility when defining the rescaled process~$X^{\eps}$. 
For example, we may define $X^{\eps}_t := \eps^{ \alpha } X_{\eps^{\gamma}t}$, 
where $\gamma:= \alpha / (\alpha /2 + 5/4)$. 
In this case define $(Z^{\eps},B^{\eps}) := \eps^{\gamma ( \alpha + 1/2)} (Z,B)$, 
and $v^{\eps}_{t} := \eps^{\alpha + \gamma } v_{\eps^{\gamma} t}$,
so that $X^{\eps}$ satisfies a LDP with speed $\eps^{-2 \gamma(\alpha + 1/2)}$ 
and rate function identical to that in Theorem~\ref{thm:rBergLDPthm}. 
This essentially falls in the category of moderate deviations, within the context of~\cite{Gui03}, 
for the original process~$X$, which is scaled by $1/(h(t)\sqrt{t})$, 
where $h(t) \in [1, t^{-1/2}]$ for small enough~$t$.

\begin{remark}\label{rem:UncorrelatedCase}
The structure of the Hilbert space~$\mathscr{H}_{\rho }^{K_{\alpha}}$ 
in Corollary~\ref{cor:RKHS2} below
precisely determines the rate function~$ \Lambda^X$.
In the uncorrelated case $\rho=0$, $\mathscr{H}_{\rho }^{K_{\alpha}}$
(and its inner product) is degenerate, 
and clearly~$\Lambda^*$ does not make sense.
This case needs to be treated separately and is analysed in Section~\ref{sec: uncorrel}.
\end{remark}

From~\eqref{eq:IiOper}, every $\zxy\in \mathscr{H}^{K_\alpha}_{\rho}$ 
has the representation $\zxy = \Ii^{K_\alpha}_{\rho}f$, 
for some $f \in \LTwo$.
Therefore the rate function in Theorem~\ref{thm:rBergLDPthm} can be rewritten as
\begin{equation} \label{eqn:reformulatedXRateFunction}  
\Lambda^X(\varphi )
 = \inf_{f \in \LTwo} \left\{\frac{1}{2} \int_0^1 f^2(u) \D u:
 \varphi = \II\left(\Mm\left(\Ii^{K_\alpha}_{\rho}f\right)\right)
  \right\}.
\end{equation}
With this formulation, it is easy to see that $\Lambda^X(0)=0$: 
denoting $\mathrm{z}^{\widetilde{x}}_{y}:= \Mm\zxy$ and using that $\widetilde{x} > 0$, 
it follows that if $\II(\widetilde{x},y)=0$, then $ y \equiv 0$, which in turn implies that $f \equiv 0,$ 
and hence $\Lambda^X(0)=0$.
Furthermore, since clearly $\Lambda^X$ cannot take negative values, 
its minimum value is attained at the origin.


\section{Gaussian Measures on Banach Spaces and Large Deviations}\label{sec:GaussLDP}
In this section, we gather several elements from the theory of Gaussian measures and large deviations
in order to prove Theorem~\ref{thm:rBergLDPthm}.
This proof shall require a certain number of steps, in particular the precise characterisation of the reproducing kernel Hilbert spaces associated to the different processes under consideration.


\subsection{Gaussian measures on Banach spaces}\label{sec:Gaussian}
A centred process~$( Z_t)_{t \in \Tt}$ is called Gaussian 
if for all $n \in \mathbb{N}$ and any $t_1,\ldots,t_n \in \Tt$, 
the random variables $Z_{t_1},\ldots, Z_{t_n}$ are jointly Gaussian;
any such process is then completely characterised by its covariance function. 
We recall some basic facts, needed later, on Gaussian measures on Banach spaces, 
mostly following Carmona and Tehranchi~\cite[Chapter 3]{CT06}.
Let $\left( \Es, \|\cdot\|_{ \Es} \right) $ be a real, separable Banach space, 
and~$\Es^{*}$ its topological dual (i.e. the space of all linear functionals on~$\Es$), 
with duality relationship $\langle \cdot,\cdot \rangle_{ \Es^{*} \Es }$. 
The bilinear functional $\langle \cdot ,\cdot \rangle_ {\Es^* \Es  } : \Es^*
\times \Es \to \RR$ is such that if $\langle x^*, x \rangle_ {\Es^* \Es}=0$ for all $x^*$ $\in \Es^*$ 
(resp. $x \in \Es$) then $x=0$ (resp. $x^* =0$)~\cite[Page 195]{AB06}. 
We shall further let $\Bs(\Es)$ denote the Borel $\sigma$-algebra of~$\Es$.

\begin{definition}\cite[Definition 3.1]{CT06} \label{defGaussian}
A measure~$\mu$ on $(\Es,\Bs(\Es))$ is (centred) Gaussian 
if every $f^* \in \Es^*$, when viewed as a random variable via the dual pairing 
$f \mapsto \langle f^*, f \rangle_{ \Es^* \Es}$, is a (centred) real Gaussian random variable on 
$(\Es,\Bs(\Es),\mu)$.
\end{definition}

The following proposition~\cite[Proposition 3.1]{CT06} characterises Gaussian measures on Banach spaces.
\begin{proposition}\label{prop1} 
Any (centred) Gaussian measure~$\mu$ on~$\Es$ is the law of some (centred) Gaussian process 
with continuous paths, indexed by some compact metric space. 
\end{proposition} 

Note that every real-valued, centred Gaussian process taking values in~$\Es$  induces some measure on~$\Cc$, 
the space of continuous functions from~$\Tt$ to~$\RR$.
By Proposition~\ref{prop1}, one can construct a centred Gaussian probability measure~$\mu$ on~$\Es$ 
by constructing the corresponding Gaussian process. 
The above argument may be extended to a $d$-dimensional centred Gaussian process,
thereby inducing a Gaussian measure on $\Es=\Cc^d$.
For a Gaussian measure~$\mu$ on~$\Es$, we introduce the bounded, linear operator $\Gamma: \Es^*\to \Es$ as
\begin{equation}\label{eq:Gamma}
\Gamma(f^*):= \int_{\Es} \langle f^*, f \rangle_{\Es^* \Es} f \mu(\D f),
\end{equation}
and note in particular that $\langle f^*, f \rangle_{\Es^* \Es} f$ 
is an $\Es$-valued random variable on 
$\left(\Es, \Bs( \Es), \mu\right)$. 

\begin{definition}\cite[Definition 3.3]{CT06}\label{defRKHS} 
The reproducing kernel Hilbert space (RKHS)~$\mathcal{H}_{\mu}$ of~$\mu$ 
is defined as the completion of~$\Gamma(\Es^*)$ with the inner product
$\displaystyle \left\langle \Gamma(f^*), \Gamma(g^*) \right\rangle_{\mathcal{H}_{\mu}}:= \int_{\Es} \langle f^*, f \rangle_{\Es^* \Es} \langle g^*, f \rangle_{\Es^* \Es} \mu(\D f)$.
\end{definition}

For the inclusion map $\iota: \mathcal{H}_{\mu} \to \Es$, 
the space $\iota(\mathcal{H}_{\mu})$ is dense in~$\Es$; 
it follows then for the adjoint map $\iota^{*}:\Es^* \to \mathcal{H}_{\mu}^{*}$ 
that $\iota^*(\Es^*)$ 
is dense in $\mathcal{H}_{\mu}^{*}$. 
Recall also that $\mathcal{H}_{\mu}$ and $\mathcal{H}_{\mu}^{*}$ are isometrically 
isomorphic, which we denote by $\mathcal{H}_{\mu}^{*} \simeq \mathcal{H}_{\mu}$ 
(by the Riesz representation theorem, 
as $\RR$ is the underlying field). 
Now, for a centred Gaussian random variable~$f^*$ on~$\Es$, by Definition~\ref{defGaussian}, it follows that
$$
\EE\left(\langle f^*, f \rangle_{\Es^*\Es}^2 \right)
 = \int_{\Es} \langle f^*, f \rangle_{\Es^*\Es} \langle f^*, f \rangle_{\Es^*\Es} \mu(\D f)
 = \| f \|_{\mathcal{H}_{\mu} }^2 
= \|\iota^{*} f^{*} \|_{ \mathcal{H}_{\mu}^*}^2.
$$
This yields the following equivalent form of Definition~\ref{defRKHS} for the RKHS of~$\mu$~\cite[Page 88]{DS89}.

\begin{definition}\label{defRKHS2} 
A real, separable Hilbert space~$\mathcal{H}_{\mu}$ such that $\mathcal{H}_{\mu} \subset \Es$
is the RKHS of~$\mu$ if the following two conditions hold: 
\begin{itemize}
\item there exists an embedding $I: \mathcal{H}_{\mu}\to\Es$, i.e. an injective continuous map whose image is  dense in~$\Es$;
\item any $f^*\in \Es^{*}$ is a centred Gaussian random variable on $\Es$ with 
variance $\| I^* f^*\|_{\mathcal{H}_{\mu}^{*} }^2$.
\end{itemize}
\end{definition}
\begin{remark}
The embedding $I$ need not necessarily be the inclusion map.
\end{remark}

\begin{remark}\label{remarkInclusion}
Given a triplet $(\Es, \mathcal{H}_{\mu}, \mu)$,
consider the inclusion map $I^*: \Es^{*} \to L^2(\Es, \mu) $ 
(we think of~$\Es^{*}$ as a dense subset in $\mathcal{H}_{\mu}^{*} \simeq \mathcal{H}_{\mu}$ 
by~$\iota^*$). 
Since~$I^*$ preserves the Hilbert space structure of $L^2(\Es,\mu)$, 
it can be extended to an isometric embedding 
$\bar{I}^* : \mathcal{H}_{\mu}^{*} \to L^2(\Es,\mu)$ such that 
$\|\bar{I}^* f^*\|_{\mathcal{H}_{\mu}^{*}} = \| f^* \|_{L^2(\Es,\mu)}$.
\end{remark}

We now explicitly characterise the RKHS of the measures 
induced by~$(Z_t)_{t \in \Tt }$ (introduced in~\eqref{eq:SDEZ}) on~$\Cc$, 
and by~$((Z_t,B_t))_{t \in \Tt}$ on~$\Cc^2$.
In fact, we first prove a more general result, using the operators in~\eqref{eq:IiOper}.
Introduce the following assumption on the function $\varphi$ defining the operator $\Ii^\varphi $:
\begin{assumption}\label{ass:Convolution}
There exists $\phi\in L^2(\Tt ,\RR)$ such that $\int_0^{\varepsilon} \vert \phi(s) \vert \D s > 0 $ for some $\varepsilon > 0$ and $\varphi(\cdot,t)=\phi(t-\cdot)$ for any $t \in \Tt$.
\end{assumption}

\begin{theorem}\label{thm:RKHS1}
Let $\varphi$ satisfy Assumption~\ref{ass:Convolution}, 
which makes~$\Ii^\varphi$ injective on~$\LTwo$. 
The RKHS of the measure induced by the process~$\int_{0}^{\cdot}\varphi(u,\cdot)\D W_u$ on~$\Cc$ is given by
$\mathscr{H}^{\varphi} = \{\Ii^\varphi f : f \in \LTwo\}$,
with inner product 
$\langle \Ii^{\varphi} f_1, \Ii^{\varphi} f_2\rangle_{\mathscr{H}^{\varphi}} := \langle f_1, f_2  \rangle_{\LTwo}$.
\end{theorem}

\begin{proof}[Proof of Theorem \ref{thm:RKHS1}]
See Section \ref{Section: proof of main results}.
\end{proof}

\begin{corollary}\label{cor:RKHS1}
The RKHS of the Gaussian measure induced (on~$\Cc$) by~$(Z_t)_{t \in \Tt}$ in~\eqref{eq:SDEZ} is
$\mathscr{H}^{K_{\alpha}}$.
\end{corollary}

We need to extend Theorem~\ref{thm:RKHS1} (and Corollary~\ref{cor:RKHS1}) to find the RKHS of  
the Gaussian measure on the space~$\Cc^2$ 
induced by the two-dimensional process $((Z_t, B_t))_{t \in \Tt}$, 
where~$Z$ and~$B$ are defined in~\eqref{eq:SDEZ} and~\eqref{rBergomi} respectively.   

\begin{theorem}\label{thm:RKHS2} 
Let $\varphi_1,\varphi_2$ satisfy Assumption~\ref{ass:Convolution}, 
which makes~$\Ii^{\varphi_1}_{\varphi_2}$ injective on~$\LTwo$.
Introduce the $\RR^2$-valued process $(Y^1, Y^2)$ as 
$Y^i := \int_0^{\cdot} \varphi_i (s, \cdot) \D W^i_s$ for $i=1,2$,
where $W^1$ and $W^2$ are two standard Brownian motions with
correlation $\rho \in [-1,1]\backslash \{0\}$. 
Then the RKHS of the measure induced by~$(Y^1,Y^2)$ on~$\Cc^2$ is
$\mathscr{H}^{\varphi_1}_{\varphi_2} = \left\{\Ii^{\varphi_1}_{\varphi_2} f: f\in \LTwo \right\}$,
with inner product
$\left\langle \Ii^{\varphi_1}_{\varphi_2} f_1, \Ii^{\varphi_1}_{\varphi_2} f_2\right\rangle_{\mathscr{H}^{\varphi_1}_{\varphi_2}} := \langle f_1, f_2 \rangle_{\LTwo}$.
\end{theorem}

\begin{proof}[Proof of Theorem \ref{thm:RKHS2}]
See Section \ref{Section: proof of main results}.
\end{proof}

\begin{corollary}\label{cor:RKHS2} 
The RKHS of the measure induced (on~$\Cc^2$) by the process~$((Z_t,B_t))_{t \in \Tt}$ is
$\mathscr{H}^{K_\alpha}_{\rho}$.
\end{corollary} 

\subsection{Large deviations for Gaussian measures}
We now concentrate on large deviations for Gaussian measures. 
As before, $\Es$ denotes a real, separable Banach space with norm $\| \cdot \|_{\Es}$, 
and we introduce a centred Gaussian measure~$\mu$ on $(\Es, \Bs(\Es))$ such that,
for any $\yy \in \Es^{*}$,
$$
\int_{ \Es} \E^{\textrm{i} \langle \yy,\xx \rangle_{ \Es^{*}\Es}}\mu(\D \xx)=
\exp\left(-\frac{C_{\mu}(\yy,\yy)}{2}\right),
$$ 
where $C_{\mu}: \Es^{*} \times \Es^{*} \to [0,+ \infty)$ is a bilinear, symmetric map. 
We define $\Lambda^*_{\mu}:\Es \to \mathbb{R} $ as
$\Lambda^{*}_{\mu}(\xx)
 :=\sup\left\{\langle \yy,\xx \rangle_{\Es^{*}\Es} - \frac{1}{2}C_{\mu}(\yy, \yy): \yy\in \Es^{*}\right\}$ on~$\Es$.
The following lemma is proved in~\cite[Lemma 3.4.2]{DS89}.

\begin{lemma}\label{DSLDPlemma} The following three statements hold for the measure $\mu$: \begin{enumerate} 
\item\label{item:fernique} There exists $\alpha>0$ such that 
$\displaystyle \int_{\Es} \exp\left(\alpha \|\xx\|_{\Es}^2\right)\mu (\D \xx)$  is finite;
\item For all $\yy\in \Es^{*}$, 
$\displaystyle C_{\mu}(\yy,\yy)= 
\int_{ \Es} \langle \yy,\xx \rangle^2_{\Es^{*}\Es} \mu(\D\xx) 
\le \|\yy\|^2_{\Es^{*}}\int_{\Es} \|\xx\|^2_{\Es} \mu(\D\xx) \in (0, + \infty)$;
\item $ \Lambda^{*}_{\mu} $ defines a rate function on $ \Es$ and satisfies 
$\Lambda^{*}_{\mu}(a\yy)=a^2 \Lambda^{*}_{\mu}(\yy)$ for all $a \in \RR$.
\end{enumerate}
\end{lemma}

\blue{The reader may recognise statement \eqref{item:fernique} in Lemma \ref{DSLDPlemma} above as Fernique's theorem \cite{Fer70}.}
For an $\Es$-valued Gaussian random variable~$X$ with distribution~$\mu$, 
define $X^{\eps}:= \eps^{1/2}X$, with law~$\mu_{\eps}$. 
Then the following holds~\cite[Theorem 3.4.5]{DS89}:

\begin{theorem}\label{DSthmLDP}
The sequence of probability measures~$(\mu_{\eps})_{\eps\ge 0}$ satisfies a LDP on~$\mathscr{E}$ 
with speed $\eps^{-1}$ and rate function~$\Lambda^{*}_{\mu}$.
\end{theorem}

\begin{remark}\label{DSremark} 
Theorem~\ref{DSthmLDP} implies in particular that the standard Brownian motion 
$(W_t)_{t \ge 0}$ satisfies a LDP  on~$\RR$ with speed $t^{-1}$,
since $W_t$ and $ \sqrt{t} W_1$ are equal in law. 
\end{remark}

\begin{corollary}\label{Corollary: Z LDP} 
For any $t\in \Tt$, let $\nu_t$ be the law of $Z_t$ defined in~\eqref{eq:SDEZ}. 
Then the sequence $(\nu_t)_{t>0}$ satisfies a LDP on~$\RR$ 
with speed $t^{-\beta} $ and rate function
$\Lambda^{*}_{\mu}(x) := \frac{x^2}{2 \eta ^2}$, for $x \in \RR$. 
\end{corollary}
\begin{proof}
Here, $\Es=\RR$ and $\langle u, v \rangle_{ \Es^{*}\Es}=uv$.
Since~$Z_t$ and~$t^{\beta/2}Z_1$ are equal in law, 
and 
$\int_{\RR}\E^{\I y x}\PP(Z_1 \in \D x) = \exp(-y^2 \eta^2/2)$,
taking $C_{\mu}(x,y) \equiv xy \eta^2$, the proof follows from Theorem~\ref{DSthmLDP} and Remark~\ref{remark: DSthmLDP}.
\end{proof}

The following two results will be essential for establishing a LDP for the rough Bergomi model. 
The first one, the contraction principle, states that continuous mappings preserve large deviations principles,
while the second one is a universal LDP result for general Gaussian measures on Banach spaces.

\begin{proposition}[Theorem 4.2.1. in~\cite{DZ10} \textbf{(Contraction Principle)}]
\label{contractionPrinciple}
Let $\Es$ and $\widetilde{ \Es}$ be two Hausdorff topological spaces 
and let $f:~ \Es~ \to ~\widetilde{ \Es }$ be a continuous mapping. 
Let $(\nu_{\eps})_{\eps\ge 0}$, $(\widetilde{\nu}_{\eps})_{\eps\ge 0}$ 
be two families of probability measures on $(\Es, \Bs(\Es) )$ and 
$(\widetilde{\Es}, \Bs( \widetilde{\Es}))$ respectively, 
such that $\widetilde{\nu}_{\eps}= \nu_{\eps} \circ f^{-1}$ for each $\eps>0$. 
If $(\nu_{\eps})_{{\eps}\ge 0}$ satisfies a LDP on~$\Es$ with rate function~$\Lambda$, 
then $\left( \widetilde{\nu}_{\eps} \right)_{ {\eps} \ge 0}$ satisfies a LDP on~$\widetilde{\Es}$ with rate function
$$
\widetilde{\Lambda}(\yy):= \inf \left\{\Lambda(\xx): \yy=f(\xx)\right\}.
$$
\end{proposition}

\begin{theorem}[Theorem 3.4.12 in~\cite{DS89}]\label{thm3.4.12}
Let $B$ be a $d$-dimensional Gaussian process, inducing a measure~$\mu$ 
on $(\Cc^d, \Bs(\Cc^d))$ with RKHS~$\mathcal{H}_{\mu}$. 
Then $ (\eps \mu)_{{\eps} \ge 0}$ satisfies a LDP with speed~$\eps^{-1}$ 
and rate function
\begin{equation*}
\Lambda_{\mu}^*(\xx) :=
\begin{cases}
\displaystyle \frac{1}{2}\|\xx\|_{\mathcal{H}_{\mu}}^2, & \text{if }\xx\in \mathcal{H}_{\mu},\\
+ \infty, & \text{otherwise.} 
\end{cases}
\end{equation*}
\end{theorem}

In order to deal with the stochastic differential equation~\eqref{eq:SDEXeps},
we have to consider the stochastic integral $\int_{0}^{\cdot} \sqrt{v^\eps_s}\D B^\eps_s$.
Assuming that the sequence $(\sqrt{v^\eps_s}, B^\eps_s)$ converges weakly as $\eps$ tends to zero
yields, under some conditions, a weak convergence for the stochastic integral~\cite{Jaku, Kurtz}.
However, in order to state a large deviations principle, we need a stronger result, proved by Garcia~\cite{Gar06}.
Before stating it (Theorem~\ref{GarciaLDPtheorem} below), though, 
we introduce the following class of sequences of stochastic processes:
\begin{definition}[Definition 1.1 in~\cite{Gar06}]\label{UETdef}
Let $\mathscr{U}$ denote the space of simple, real-valued, adapted processes~$Z$ 
such that $\sup_{t\ge 0} |Z_t|\le 1$.
A sequence of semi-martingales $(Y^{\eps})_{\eps\ge 0}$ is said to be uniformly exponentially tight (UET) if, 
for every $c, t>0$, there exists $K_{c,t}>0$ such that 
$$
\limsup_{\eps \downarrow 0} \eps \log \left(\sup_{Z \in \mathscr{U}} 
\PP\left( \sup_{s \le t} \left\vert \int_0^s Z_{u-} \D Y^{\eps}_u  \right\vert  \ge K_{c,t} \right) \right) \leq -c.
$$ 
\end{definition}

\begin{theorem}[Theorem 1.2 in~\cite{Gar06}]\label{GarciaLDPtheorem} 
Let $(X^{\eps})_{\eps\ge 0}$ be a sequence of adapted, c\`adl\`ag stochastic processes, 
and $(Y^{\eps})_{\eps\ge 0}$ a sequence of uniformly exponentially tight semi-martingales. 
If the sequence $((X^{\eps}, Y^{\eps}))_{\eps\ge 0}$ satisfies a LDP with rate function~$\Lambda$, 
then the sequence of stochastic integrals
$(X^{\eps}\cdot Y^{\eps})_{\eps\ge 0}$ 
satisfies a LDP with rate function 
$\widehat{\Lambda} (\varphi):= \inf \left\{\Lambda(\zxy): 
\varphi = \xx\cdot \yy, \yy\in \BV \right\}$.
\end{theorem}

\section{Proof of the main results}\label{Section: proof of main results}
\begin{proof}[Proof of Theorem~\ref{thm:RKHS1}]
Let Assumption~\ref{ass:Convolution} hold for a given function $\varphi\in\LTwo$.
The operator~$\Ii^{\varphi}$ in~\eqref{eq:IiOper}
is surjective on~$\mathscr{H}^{\varphi}$.
Let $f_1, f_2 \in \LTwo$ be such that $\Ii^{\varphi}f_1 = \Ii^{\varphi}f_2$. 
Then $\int_{0}^{t} \varphi(u,t)[f_1(u) - f_2(u)]\D u=0$ for any $t\in\Tt$. 
Titchmarsh's convolution theorem~\cite[Theorem VII]{Tit26} then implies that $f_1=f_2$ almost everywhere,
so that ~$\Ii^{\varphi}$ is a bijection.
The linearity  of~$\Ii^{\varphi}$ implies that 
$\langle \Ii^{\varphi}f_1, \Ii^{\varphi}f_2\rangle_{\mathscr{H}^{\varphi}}:= \langle f_1, f_2 \rangle_{\LTwo}$
defines an inner product on~$\mathscr{H}^{\varphi}$, and hence 
$(\mathscr{H}^{\varphi}, \langle \cdot, \cdot \rangle_{\mathscr{H}^{\varphi}})$ 
is a real inner product space.
In order for~$\mathscr{H}^{\varphi}$ to satisfy Definition~\ref{defRKHS2}, 
we first need to show that it is a separable Hilbert space.
Let $\{ f_n \}_{n\in\NN}$ be a sequence of functions in~$\LTwo$ such that
$\{\Ii^{\varphi} f_n\} _{n \in \mathbb{N} }$ converges to~$\Ii^{\varphi}f$ in $\mathscr{H}^{\varphi}$.
Therefore
$\|\Ii^{\varphi}f_n - \Ii^{\varphi}f_m\|_{\mathscr{H}^{\varphi}} = \|f_n -f_m\|_{\LTwo}$ 
tends to zero as~$n$ and~$m$ tend to infinity.
Since~$\LTwo$ is a complete (Hilbert) space, there exists a function $\widetilde{f} \in \LTwo$
such that the sequence $\{ f_n \}_{n \in \mathbb{N} } $ converges to~$\widetilde{f}$. 
Assume for a contradiction that $f \neq \widetilde{f}$, then, since~$\Ii^{\varphi}$ is a bijection,
the triangle inequality yields
$$
0 < \left\|\Ii^{\varphi}f - \Ii^{\varphi}\widetilde{f} \right\|_{\mathscr{H}^{\varphi}}
\le\left\|\Ii^{\varphi}f - \Ii^{\varphi}f_n \right\|_{\mathscr{H}^{\varphi}}
 +\left\|\Ii^{\varphi}\widetilde{f} - \Ii^{\varphi}f_n\right\|_{\mathscr{H}^{\varphi}},
$$
which converges to zero as~$n$ tends to infinity.
Therefore $f=\widetilde{f}$, $\Ii^{\varphi}f \in\mathscr{H}^{\varphi}$ and $\mathscr{H}^{\varphi}$ is complete, 
hence a real Hilbert space. 
Since~$\LTwo$ is separable with countable orthonormal basis $\{\phi_n\}_{n \in \mathbb{N}}$,
then $\{\Ii^{\varphi} \phi_n\}_{n \in \mathbb{N}}$ is an orthonormal basis for~$\mathscr{H}^{\varphi}$, 
which is then separable.\\
We now find a dense embedding $I: \mathscr{H}^{\varphi} \to \Cc$ as in Definition~\ref{defRKHS2}.
Since $\mathscr{H}^{\varphi} \subset \Cc$, take the embedding to be the inclusion map $I=\iota$. 
By~\cite[Lemma 2.1]{Che08}, the conditions on~$\phi$ in Assumption~\ref{ass:Convolution}
imply that~$\mathscr{H}^{\varphi}$ is dense in~$\Cc$.
Finally, for $f^*\in\Cc^*$, the measure~$\mu$ induced by the process~$\int_{0}^{\cdot}\varphi(u,\cdot)\D W_s$ is a Gaussian probability measure on 
$(\Es, \Bs(\Es))$, 
and $f^*$ is a centred, real Gaussian random variable on $(\Es,\Bs(\Es),\mu)$ 
by Definition~\ref{defGaussian} . 
In turn, Remark~\ref{remarkInclusion} implies that~$I^*$, the dual of~$I$,
admits an isometric embedding~$\bar{I}^*$ such that
$\|\bar{I}^* f^*\|_{(\mathscr{H}^{\varphi})^{*}}^2
 = \| f^*\|_{ L^2(\Es,\mu) }^2 = \int_{\Es }(f^*)^2 \D \mu = \VV(f^*)$,
hence $\mathscr{H}^{\varphi}$ is the RKHS of~$\mu$. 
\end{proof}

\begin{proof}[Proof of Theorem~\ref{thm:RKHS2}]  
We proceed in a similar manner to the proof of Theorem~\ref{thm:RKHS1}.
Let $\varphi_1, \varphi_2$ satisfy Assumption~\ref{ass:Convolution}.
Clearly the operator~$\Ii^{\varphi_1}_{\varphi_2}$ in~\eqref{eq:IiOper} is surjective 
on $\mathscr{H}^{\varphi_1}_{\varphi_2} \subset \Cc^2$. 
By Titchmarsh's convolution Theorem~\cite[Theorem VII]{Tit26}, 
if $\Ii^{\varphi_1}_{\varphi_2} f_1 = \Ii^{\varphi_1}_{\varphi_2} f_2$ 
on~$\Tt$, then $f_1=f_2$ and~$\Ii^{\varphi_1}_{\varphi_2}$ is a bijection.
Furthermore 
$\langle \Ii^{\varphi_1}_{\varphi_2}f_1, \Ii^{\varphi_1}_{\varphi_2}f_2 \rangle_{\mathscr{H}^{\varphi_1}_{\varphi_2}} := \langle f_1, f_2 \rangle_{\LTwo}$ 
is a well-defined inner product and, following the proof of Theorem~\ref{thm:RKHS1},
$(\mathscr{H}^{\varphi_1}_{\varphi_2}, \langle \cdot, \cdot \rangle_{\mathscr{H}^{\varphi_1}_{\varphi_2}} )$ 
is a real, separable Hilbert space.
To find a dense embedding $I: \mathscr{H}^{\varphi_1}_{\varphi_2} \to \Cc^2$, 
take~$I$ as the inclusion map~$\iota$; then the conditions on~$\phi_1,\phi_2$ in Assumption~\ref{ass:Convolution} imply that~$\mathscr{H}^{\varphi_1}_{\varphi_2}$ is dense in~$\Cc^2$ by~\cite[Lemma 2.1]{Che08}.
Finally, $f^{*} \in (\Cc ^{2})^* $ is a real, centred Gaussian random variable on
$(\Cc^2, \Bs(\Cc ^2 ), \mu)$, where~$\mu$ denotes the measure induced by~$(Y^1,Y^2)$, and 
$
\VV(f^*)= \int_{\Cc^2 } (f^*)^2 \D \mu_2 = \| f^{*}\|_{L^2(\Cc^2, \mu_2)}^2
 = \|\iota^* f^*\|_{(\mathscr{H}^{\varphi_1}_{\varphi_2})^{*}}^2$,
so that by Definition~\ref{defRKHS2}, $\mathscr{H}^{\varphi_1}_{\varphi_2}$ is the RKHS of the measure induced by~$(Y^1,Y^2)$ on~$\Cc^2$.
\end{proof}

\begin{proof}[Proof of Theorem~\ref{thm:rBergLDPthm}] 
Let $(Z^{\eps},B^{\eps})$ be as in~\eqref{def: rescaled rBerg processes}.
From Theorem~\ref{thm3.4.12} and Corollary~\ref{cor:RKHS2} 
the sequence $((Z^\eps, B^\eps))_{\eps\ge 0}$ 
satisfies a LDP with speed $\eps^{-\beta}$ and rate function
(with $\mathscr{H}^{K_\alpha}_{\rho}$ given in Corollary~\ref{cor:RKHS2})
\begin{equation*}
\Lambda^{*}(\zxy)
 = 
\begin{cases}
\displaystyle \frac{1}{2} \left\|\zxy\right\|^2_{\mathscr{H}^{K_\alpha}_{\rho}}, & \text{if }
\zxy\in \mathscr{H}^{K_\alpha}_{\rho}, \\
+ \infty, & \text{otherwise}.
\end{cases}
\end{equation*} 
Pathwise, we may view the map
$t \mapsto (Z^{\eps}_t, B^{\eps}_t)^\top$ 
as an element of $\Cc^2$, and write 
$$\begin{pmatrix}
v^{\eps}_t \\
B^{\eps}_t
\end{pmatrix}
 =  \Mm\begin{pmatrix}Z^{\eps}\\ B^{\eps}\end{pmatrix}(t, \eps).$$
 
We first verify that the operator~$\Mm$ in~\eqref{def: M operator} is continuous 
with respect to the $\Cc(\Tt \times \RR_+, \RR_{+} \times \RR)$ 
norm~$\|\cdot\|_{\infty }$. 
For any $(f, g)^\top\in \Cc^2$, introduce a small perturbation $(\delta^f, \delta^g) \in \Cc^2$.
Then
\begin{align*}
\left\|\Mm
\begin{pmatrix}
f+\delta^f \\ 
g + \delta^g
\end{pmatrix}
 - \Mm \begin{pmatrix}
f \\ 
g 
\end{pmatrix}\right\|_{\infty} 
& = \sup_{ t \in \Tt, \eps>0} \left\{
\Big|(\mm (f+\delta^f))(t,\eps) - (\mm f)(t,\eps)\Big| + |\delta^g(t)|\right\}\\
& \leq \sup_{ t\in\Tt, \eps>0}\left\{v_0^{1+\beta} \exp\left(-\frac{\eta^2}{2} (\eps t)^{\beta }\right)
\left|\E^{f(t)}\right| \left|\E^{\delta^f(t)} - 1\right|\right\} + \sup_{t \in \Tt }|\delta^g(t)| \\
& \leq C\sup_{t \in\Tt} \left|\E^{\delta^f(t)} - 1\right| + \sup_{t \in \Tt }|\delta^g(t)|,
\end{align*} 
for some strictly positive constant~$C$.
The right-hand side clearly tends to zero as $(\delta^f, \delta^g)$ 
tends to zero with respect to the sup norm on~$\Cc^2$, and hence~$\Mm$ is a continuous operator.
The \blue{c}ontraction \blue{p}rinciple (Proposition~\ref{contractionPrinciple}) therefore implies that the sequence 
$((v^{\eps}, B^{\eps}))_{\eps\ge 0}$ satisfies a LDP on $\Cc(\Tt \times \RR_+, \RR_{+} \times \RR)$, 
with speed~$\eps^{-\beta}$ and rate function~$\Lambda$.
Since~$\Mm$ is clearly a bijection, the rate function~$\Lambda$ 
may then be expressed as $\Lambda(\zxyO)= \Lambda^{*}\left(\Mm^{-1}(\zxyO)\right)$,
for any $(x_1,y_1) \in \Cc^2$.

In the second step we will apply Theorem~\ref{GarciaLDPtheorem} to prove that
the sequence of stochastic integrals
$(\II(v^{\eps},B^{\eps})(\cdot))_{\eps \ge 0 } := (\int_0^{\cdot} \sqrt{v^{\eps}_s}\D B^{\eps}_s)_{\eps \ge 0 }$ 
satisfies a LDP. 
Since $B^\eps = \eps^{\alpha+1/2}B$ by~\eqref{eq:SDEXeps}, we can write the stochastic integral
$\II(v^{\eps},B^{\eps})(\cdot) = \II(\eps^{2 \alpha} v^{\eps},\sqrt{\eps}B)(\cdot)$, which holds almost surely, 
and therefore~\cite[Example 2.1]{Gar06} the sequence of (semi)-martingales $(\sqrt{\eps} B)_{\eps \ge 0}$ 
is UET in the sense of Definition~\ref{UETdef}. 
Since the sequence $(\sqrt{\eps^{2 \alpha} v^{\eps}})_{\eps\ge 0}$ consists of c\`adl\`ag, 
$(\Ff_t)$-adapted processes, 
Theorem~\ref{GarciaLDPtheorem} implies that the sequence of stochastic integrals
$( \II(v^{\eps},B^{\eps})(\cdot))_{\eps\geq 0}$ 
satisfies a LDP with speed~$\eps^{-\beta}$ and rate function 
\begin{equation*}
\Lambda^X(z)= \inf \{
\Lambda(\zxy),  z=\II(x,y) \text{ and }  y\in\BV\cap\Cc \}.
\end{equation*}

The final step is to prove the LDP for
$X^{\eps} = \int_0^{\cdot} \sqrt{v^{\eps}_s}\D B^{\eps}_s - \frac{1}{2}\int_0^{\cdot} v^{\eps}_s \D s$.
To do this we show that the sequences $(X^{\eps})_{\eps \ge 0}$ and $(\II(v^{\eps},B^{\eps}))_{\eps \ge 0} $ are exponentially equivalent. 
For any $\delta > 0 $ it follows that
$$
\PP \left( \sup_{t \in [0,1]}\vert X^{\eps}_t - \II(v^{\eps},B^{\eps})(t)  \vert > \delta \right)  \le 
\PP \left( \int_0^1 v^{\eps}_s \D s > \delta \right) 
 \le 
\PP \left( \int_0^1 \exp(Z^{\eps}_s) \D s > b_{\eps} \right),
$$
where $b_{\eps}:= \delta / v_0\eps^{1+\beta}$.
Using that 
$\int_0^1 \exp( Z^\eps_s) \D s 
\le \exp( \sup_{t \in [0,1]} Z^\eps_t )
$
almost surely,
it follows that 
$$ \PP \left( \int_0^1 \exp( Z^\eps_s ) \D s > b_{\eps} \right)
\le 
\PP \left(  \sup_{t\in [0,1]}Z^\eps_t > \log b_{\eps} \right)
=
\PP \left( \sup_{t \in [0,1]} Z_t > \frac{\log b_{\eps} }{ \eps^{\beta/2}} \right). 
$$
The process $ (Z_t)_{t \in [0,1]} $ is almost surely bounded~\cite[Theorem 1.5.4]{AT07}, 
and 
so we may apply the Borell-TIS inequality; 
a consequence of which~\cite[Theorem 2.1.1 and discussion below]{AT07}, 
implies that 
$$
\PP \left( \sup_{t \in [0,1]} Z_t > \frac{\log b_{\eps} }{ \eps^{\beta/2}} \right)
\le 
\exp \left\{ -\frac{1}{2}\left( \frac{\log b_{\eps} }{ \eps^{\beta/2}} - \EE\left(\sup_{t\in [0,1]}Z_t\right) \right)^2 \right\}.
$$
This then implies that 
$$ \eps^{\beta} \log \PP \left( \int_0^1 \exp( Z_s) \D s > b_{\eps} \right)
\le 
\eps^{\beta} \left\{
- \frac{(\log b_{\eps})^2}{2 \eps^{\beta}}
+ \frac{\log b_{\eps} }{ \eps^{\beta/2}}\EE\left(\sup_{t\in [0,1]}Z_t\right)
- \frac{1}{2} \EE\left(\sup_{t\in [0,1]}Z_t\right)^2
\right\}.
$$
Note that 
$ \eps^{\beta /2} \log b_{\eps}$ converges to zero
as~$\eps$ tends to zero, which in turn implies that
$$ \limsup_{\eps \downarrow 0}  \eps^{\beta /2} \log b_{\eps} \EE\left(\sup_{t\in [0,1]}Z_t\right)=0.
$$
Similarly, $ \limsup_{\eps \downarrow 0} \eps^{\beta} \EE\left(\sup_{t\in [0,1]}Z_t\right)^2=0.
$
Furthermore, it follows that
$$ \limsup_{\eps \downarrow 0} \eps^{\beta} \left( - \frac{(\log b_{\eps})^2}{2  \eps^{\beta}}  \right) = -\infty,
$$
and therefore
$$ \limsup_{\eps \downarrow 0} \eps^{\beta} \log \PP 
\left( \sup_{t \in [0,1]}\vert X^{\eps}_t - \II(v^{\eps},B^{\eps})(t)  \vert > \delta \right) = -\infty,
$$
which is precisely the definition of exponential equivalence~\cite[Definition 4.2.10]{DZ10}. 
Then, by~\cite[Theorem 4.2.13]{DZ10}, the sequence $(X^{\eps})_{\eps \ge 0}$ satisfies a LDP with speed $\eps^{-\beta} $ and rate function $\Lambda^X$.
\end{proof}

\begin{proof}[Proof of Theorem ~\ref{DSthmLDP}]  
Let $\underline{X}:=(X^1,\ldots,X^n)$ be an $n$-dimensional random vector taking values on~$\mathscr{E}^n$, where each~$X^k$ has distribution~$\mu$, so that the average $\frac{1}{n} \sum_{k=1}^n X^k$ 
has distribution~$\mu_{1/n}$.
Lemma~\ref{DSLDPlemma} implies that
$\int_{\Es} \exp(\alpha \| x\|^2_{\Es})\mu_{1/n}(\D x)$ is finite
for some $\alpha>0$, and~\cite[Theorem 3.3.11]{DS89} yields a LDP for the sequence
$(\mu_{1/n})_{n \ge 1}$, with rate function~$\Lambda^{*}_{\mu}$.
Define now $n({\eps}):= \left\lfloor \frac{1}{{\eps}} \right\rfloor \vee 1$ and $ \ell({\eps}):= {\eps} n({\eps})$ for ${\eps}>0$, noting that 
$\ell (\eps) \in [1-\eps, 1)$ for $\eps \in (0,1/2)$ and 
in~$[\frac{1}{2}, 1]$ for $\eps \in [1/2,1]$; 
for a Gaussian random variable~$X$ with distribution~$\mu_{1/n({\eps})}$, it follows that $ \ell({\eps})^{1/2}X$ has distribution
$\mu_{\eps}$.
For a closed subset~$B$ of~$\Es$, we define the dilated set
$\widetilde{B}:=\left\{ \ell^{-1/2}x: \textrm{ for all } \ell \in \left[ \frac{1}{2},1 \right], \: x\in B \right\}$, 
so that
\begin{align*}
\limsup_{\eps \downarrow 0} {\eps} \log \mu_{\eps}(B)
 & = \limsup_{\eps \downarrow 0} \frac{\ell(\eps)}{n(\eps)}
 \log\mu_{1/n(\eps)}\left(\ell({\eps})^{-1/2}B\right)\\
& \le \limsup_{\eps\downarrow 0} \frac{1}{n(\eps)} \log \mu_{1/n(\eps)}( \widetilde{B})
 = \limsup_{n \uparrow \infty} \frac{1}{n} \log \mu_{1/n}(\widetilde{B})
 \le - \inf_{ x\in \widetilde{B}} \Lambda^*_{\mu}(x).
\end{align*}
The large deviations upper bound then follows from the obvious equalities
$$
\inf_{x \in \widetilde{B}} \Lambda^{*}_{\mu}(x)
=\inf_{ \ell \in \left[\frac{1}{2},1\right]} \inf_{x \in B} \Lambda^{*}_{\mu}( \ell ^{-1/2} x) 
= \inf_{ \ell \in \left[\frac{1}{2},1\right]} \ell^{-1} \inf_{x \in B} 
\Lambda^{*}_{\mu}(x)
= \inf_{x \in B} \Lambda^{*}_{\mu}(x).
$$ 
Now for any $x$ in any open set $C \subset \Es$, 
we can find an open neighbourhood~$\Oo_x$ such that $\Oo_x \subseteq \ell({\eps})^{-1/2}C$ 
for all $0<{\eps}<{\eps}_0$ with ${\eps}_0 \in \left(0, \frac{1}{2} \right]$. 
The large deviations lower bound then follows from the inequalities
\begin{align*}
\liminf_{\eps \downarrow 0} \eps \log \mu_{\eps}(C)
 & = \liminf_{\eps \downarrow 0} \frac{\ell(\eps)}{n(\eps)} 
 \log\mu_{\frac{1}{n(\eps)}}\left(\ell(\eps)^{-1/2} C\right)
 \ge  \liminf_{n \uparrow \infty} \frac{1}{n} \log \mu_{\frac{1}{n}}(\Oo_x)
 \ge -\inf_{y \in {\Oo_x}} \Lambda^{*} _{\mu}(y)
 \ge - \Lambda^{*}_{\mu}(x).
\end{align*}
\end{proof}

\begin{remark}\label{remark: DSthmLDP}
The proof of Theorem~\ref{DSthmLDP} stills holds for the case where $t^{\beta/2}X \sim \mu_t$ with speed $t^{-\beta}  $, and the proof can be easily adapted to confirm this case.
\end{remark}

\section{Large deviations for the uncorrelated Rough Bergomi model}\label{sec: uncorrel}

We treat here the special case of~\eqref{rBergomi}, 
where the Brownian motions~$W$ and~$B$ are independent ($\rho=0$). 
Following similar arguments to Corollary~\ref{cor:RKHS2}, 
and mimicking~\eqref{eq:IiOper},
we introduce the $\LTwo $ operator $\Ii^0 $ as 
$$
\Ii^0 (f_1,f_2) :=\begin{pmatrix}
\displaystyle \Ii^{K_\alpha} f_1\\
\displaystyle \Ii^1 f_2
\end{pmatrix},
\qquad\text{for any }f_1,f_2\in\LTwo,
$$
so that the RKHS (on $\Cc^2$) 
of the measure induced by~$(Z,B)$ is
$\mathscr{H}:= \left\{\Ii^0(f_1,f_2): f_1, f_2 \in \LTwo\right\}$,
with inner product 
$\left\langle
\Ii^0(f_1,f_2),
\Ii^0(g_1,g_2)\right\rangle_{\mathscr{H}}
 := \langle f_1, g_1 \rangle_{\LTwo} + \langle f_2, g_2 \rangle_{\LTwo}$,
for any $f_1, f_2, g_1, g_2 \in \LTwo$.
Similarly to Theorem~\ref{thm:rBergLDPthm},
\cite[Theorem 3.4.12]{DS89} yields a LDP on~$\Cc^2$ for $((Z^\eps, B^{\eps}))_{\eps\ge 0}$ 
with speed~$\eps^{-\beta}$ and rate function 
\begin{equation*}
\Lambda(\zxy):= 
\left\{
\begin{array}{ll}
\displaystyle \frac{1}{2} \left\|\zxy\right\|^2_{\mathscr{H}}, & \text{if} (x,y)^{\top} \in \mathscr{H}, \\
+ \infty, & \text{otherwise}.
\end{array}
\right.
\end{equation*}
This in turn yields a LDP for $((v^{\eps}, B^{\eps}))_{\eps \ge 0}$ in~\eqref{def: rescaled rBerg processes}
on~$\Cc^2$ with speed~$\eps^{-\beta}$ and rate function
$\widetilde{\Lambda}(\zxy) := 
\inf \left\{\Lambda(\zxyS): \zxy = \Mm\zxyS\right\}$,
where the operator~$\Mm$ is defined in~\eqref{def: M operator}.
In the same vein as Theorem~\ref{thm:rBergLDPthm}, Theorem~\ref{GarciaLDPtheorem} yields a LDP for
$(\int_0^{\cdot} \sqrt{v^{\eps}_s} \D B^{\eps}_s )_{\eps\ge 0}$
on~$\Cc$ with speed~$\eps^{-\beta}$ 
and rate function~$\widehat{\Lambda}^X$, defined as 
\begin{align*}
\widehat{\Lambda}^X(\varphi)
 & := \inf \left\{\widetilde{\Lambda}(\zxy) : \varphi = x \cdot y, y\in\BV\cap\Cc\right\}
 = \inf \left\{\Lambda(\zxyS) : \varphi= x \cdot y,
 \zxy = \Mm\zxyS, x^*, y^* \in \mathscr{H}\right\}\\
 & = \inf \left\{ \Lambda(\zxy): \varphi= x \cdot y, 
\zxy = 
 \Mm(\Ii^0 (f_1,f_2)), 
f_1, f_2 \in \LTwo\right\}\\
 & = \inf_{f_1, f_2 \in \LTwo} 
 \left\{ 
 \frac{1}{2} \| f_1 \|^2_{\LTwo} + \frac{1}{2} \| f_2 \|^2_{\LTwo}:  
\varphi = \int_0^{\cdot } \sqrt{\mm\left((\Ii^{K_\alpha}f_1)(s) \right)} f_2(s) \D s 
\right\}. 
\end{align*}
with~$\mm$ introduced in~\eqref{eq:M1}.
Following \blue{an} argument \blue{identical} to \blue{the one} presented in \blue{the proof of} Theorem \ref{thm:rBergLDPthm}, 
we conclude that~$(X^{\eps})_{\eps >0}$ satisfies a LDP with speed~$\eps^{-\beta}$ 
and rate function~$\widehat{\Lambda}^X$.


\end{document}